\documentclass[ctagsplt,12pt]{amsart}
\RequirePackage{newlfont}
\usepackage{amscd}
\usepackage{amsmath}
\usepackage{amsfonts}
\usepackage{amssymb}
\usepackage{enumerate}
\usepackage{graphicx}
\usepackage{latexsym}
\usepackage{color}
\numberwithin{equation}{section}

\begin{document}

\centerline{                   }

\title[Mann Iteration Process for Monotone Mappings]{Mann Iteration Process for Monotone  Nonexpansive Mappings with a Graph}
\author{ M.R. Alfuraidan /To appear in Georgian Mathematical Journal}

\address{Monther Rashed Alfuraidan\\Department of Mathematics \& Statistics \\
King Fahd University of Petroleum and Minerals\\
Dhahran 31261, Saudi Arabia.}
\email{monther@kfupm.edu.sa}
\subjclass[2010]{Primary 46B20, 47E10}

\subjclass[2010]{Primary 06F30, 46B20, 47E10}
\keywords{ Directed graph, fixed point, Mann iteration process, monotone mapping, nonexpansive mapping, Opial condition, uniformly convex Banach space, weakly connected graph.}
\maketitle

\newtheorem{theorem}{Theorem}[section]
\newtheorem{acknowledgement}{Acknowledgement}
\newtheorem{algorithm}{Algorithm}
\newtheorem{axiom}{Axiom}[section]
\newtheorem{case}{Case}
\newtheorem{claim}{Claim}
\newtheorem{conclusion}{Conclusion}
\newtheorem{condition}{Condition}
\newtheorem{conjecture}{Conjecture}
\newtheorem{corollary}{Corollary}[section]
\newtheorem{criterion}{Criterion}
\newtheorem{definition}{Definition}[section]
\newtheorem{example}{Example}[section]
\newtheorem{exercise}{Exercise}
\newtheorem{lemma}{Lemma}[section]
\newtheorem{notation}{Notation}
\newtheorem{problem}{Problem}
\newtheorem{proposition}{Proposition}[section]
\newtheorem{remark}{Remark}[section]
\newtheorem{solution}{Solution}
\newtheorem{summary}{Summary}

\thispagestyle{empty}

\begin{abstract}
Let $(X,\|.\|)$ be a Banach space.  Let $C$ be a nonempty,
bounded, closed, and convex subset of $X$ and $T: C \rightarrow C$
be a $G$-monotone nonexpansive mapping.  In this work, it is shown
that the Mann iteration sequence defined by
$$x_{n+1} = t_n T(x_n) + (1-t_n)x_n, \; n = 1, 2, \cdots$$
can be proved the existence of a fixed point of $G$-monotone nonexpansive
mappings.
\end{abstract}

\section{Introduction}
Banach's Contraction Principle \cite{banach} is remarkable in its simplicity, yet it is perhaps the most widely applied fixed point theorem in all of analysis. This is because the contractive condition on the mapping is simple and easy to test in a complete metric space, it finds almost canonical applications in the theory of differential and integral equations.  Over the years, many mathematicians successfully extended this fundamental theorem.  \\

 Nonexpansive mappings are those mappings which have Lipschitz constant equal to one.  Their investigation remain a popular area of research in various fields. In 1965, Browder \cite{browder} and G\"{o}hde \cite{gohde} independently proved that every nonexpansive selfmappings of a closed convex  and bounded subset of uniformly convex Banach space has a fixed point. This result was also obtained by Kirk \cite{kirk1965} under   slightly weaker assumptions. Since then several fixed point theorems for nonexpansive mappings in Banach spaces have been  derived \cite{khamsi-kirk}.\\

Recently a new direction has been developed when the Lipschitz condition is satisfied only for comparable elements in a partially
ordered metric space.  This direction was initiated by Ran and Reurings \cite{RR04} (see also \cite{NR-L05})
who proved an analogue of the classical Banach contraction principle in partially ordered metric spaces and by Jachymski \cite{Jachymski} in metric spaces with a graph .  The motivation of such new direction is the problem of the existence of a solution which is
positive.  In other words, the classical approaches only deal with the existence of solutions while here we ask whether a positive or
negative solution exists. It is a natural question to ask since  most of the classical metric spaces are endowed with a natural partial order.  \\

When we relax the contraction condition to the case of the Lipschitz constant equal to one, i.e., nonepxansive mapping,  the completeness of the distance will not be enough as it was done in the original case.  We need some geometric assumptions to be added.  But in general the Lipschitz condition on comparable elements is a weak assumption.  In particular, we do not have the continuity property.  Therefore one has to be very careful when dealing with such mappings.  In this paper, we use the iterative methods \cite{goebel-kirk1983} to prove the existence of fixed points of $G$-monotone nonexpansive mappings. \\

For more on metric fixed point theory, the reader may consult the books  \cite{goebel-kirk-book, khamsi-kirk}. This work was inspired by \cite{BA}.

\section{Graph Basic Definitions}
The terminology of graph theory instead of partial ordering gives a wider and clearer picture. In this section, we give the basic graph theory definitions and notations which will be used throughout.

A graph $G$ is an ordered pair $(V(G) ,E(G))$ where $V(G)$ is a set and $E(G)$ is a binary relation on $V(G)$. Elements of $E(G)$ are called edges. We are concerned here with directed graphs (digraphs) that have a loop at every vertex (i.e., $(a, a) \in E(G)$ for each $a\in V(G) $). Such digraphs are called reflexive. In this case $E(G) \subseteq V(G) \times V(G)$ corresponds to a reflexive (and symmetric) binary relation on $V$. Moreover, we may treat $G$ as a weighted graph by assigning to each edge the distance between its vertices. By $G^{-1}$ we denote the conversion of a graph $G$, i.e., the graph obtained from $G$ by reversing the direction of edges. Thus we have
$$E(G^{-1})=\{(y,x)\hspace{.01in}|\hspace{.01in} (x,y)\in E(G)\}.$$

A digraph $G$ is called an oriented graph if whenever $(u,v)\in E(G)$, then $(v,u)\notin E(G)$. The letter $\widetilde{G}$ denotes the undirected graph obtained from $G$ by ignoring the direction of edges. Actually, it will be more convenient for us to treat $\widetilde{G}$ as a directed
graph for which the set of its edges is symmetric. Under this convention,
$$E(\widetilde{G})=E(G)\cup E(G^{-1}).$$

Given a digraph $G$, a (di)path of $G$ is a sequence $a_0, a_1, . . . , a_{n}, \dots$ with
$(a_{i}, a_{i+1} )\in E(G)$ for each $i = 0, 1, 2, \dots$.  A finite path $(a_0, a_1, . . . , a_{n})$ is said to have length $n+1$, for $n \in \mathbb{N}$. A digraph is connected if there is a finite (di)path joining any two of its vertices and it is weakly connected if $\widetilde{G}$ is connected.\\

\begin{definition}
A digraph $G$ is transitive if
$$ (x,y)\in E(G)\hspace{.1in}\text{and}\hspace{.1in} (y,z)\in E(G) \Rightarrow (x,z)\in E(G) \hspace{.1in}\text{for all}\hspace{.1in} \ x,y,z\in V(G).$$
\end{definition}

\noindent Note that the transitivity of a graph $G$ does not necessarily imply that the absence of loops.  It is easy to come up with a transitive graph $G$ with loops.  Such graph will not be generated by a partial order.\\

Throughout this paper, $(X, \|.\|)$ will denote a Banach vector space.  It is well known that we have two topologies defined on $X$, mainly the strong topology and the weak topology.  For more on these topologies we refer to the book \cite{beauzamy}.

\begin{definition} Let $(X, \|.\|)$ be a Banach space.  An element $x$ is called a weak-cluster point of a sequence $(x_{n})_{n\in \mathbb{N}}$ in $X$, if there exists a subsequence $(x_{\phi(n)})_{n\in \mathbb{N}}$ such that $(x_{\phi(n)})_{n\in \mathbb{N}}$ converges weakly to $x$.  In this case, we will write
$$ weak-\lim_{n \rightarrow +\infty} x_{\phi(n)} = x.$$
\end{definition}

\noindent As Jachymski \cite{Jachymski} did, we introduce the following property:
\\ \\
Let $G$ be a reflexive digraph defined on $X$. We say that $E(G)$ has Property (P) if
\begin{itemize}
  \item[(P)] for any sequence $(x_n)_{n\in \mathbb{N}}$ in $X$ such that $(x_n, x_{n+1})\in E(G)$ for $n \geq 1$ and  $x$  is a weak-cluster point of $(x_n)_{n\in \mathbb{N}}$, there exists a subsequence $(x_{\phi(n)})_{n\in \mathbb{N}}$ which converges weakly to $x$ and $(x_{\phi(n)}, x)\in E(G)$, for every $n \geq 1$.
\end{itemize}

Note that property (P) is precisely Nieto et al. \cite{NR-L05} hypothesis relaxing continuity
assumption rephrased in terms of edges. Moreover, if $G$ is a reflexive transitive digraph defined on $X$, then the Property (P) implies the following property:
\begin{itemize}
  \item[(PT)] for any sequence $(x_n)_{n\in \mathbb{N}}$ in $X$ such that $(x_n, x_{n+1})\in E(G)$ for $n \geq 1$ and  $x$  is a weak-cluster point of $(x_n)_{n\in \mathbb{N}}$, we have $(x_{n}, x)\in E(G)$, for every $n \geq 1$.
\end{itemize}

In the sequel, we assume that $G$ is a reflexive digraph defined on $X$. Moreover,
we assume that $E(G)$ has property (P). The linear convexity structure of $X$ is assumed to be compatible with the graph structure in the following sense:

\begin{itemize}
\item[(CG)] If $(x,y)\in E(G)$ and  $(w,z) \in E(G)$, then
$$(\alpha\ x + (1-\alpha)\ w,\alpha\ y + (1-\alpha)\ z)\in E(G)$$
 for all $x,y,w,z\in X$ and $\alpha \in [0,1]$.
\end{itemize}

 Next we give the definition of $G$-monotone nonexpansive mappings.

\begin{definition}\label{monotone-nonexpansive}  Let $C$ be a nonempty subset of $X$ and $G$ be a reflexive digraph defined on $X$.  A mapping $T: C \rightarrow C$ is said to be
\begin{itemize}
\item[(1)]  $G$-monotone if $T$ is edge preserving, i.e., $(T(x),T(y))\in E(G)$ whenever $(x,y)\in E(G)$, for any $x, y \in C$.
\item[(2)]  $G$-monotone K-Lipschitzian, $K\in \mathbb{R}^{+}$, if  $T$ is $G$-monotone and
$$\|T(y)-T(x)\| \leq K\ \|y-x\|$$ for any $x,y\in C$ such that $(x,y)\in E(G)$.
\end{itemize}
If $K =1$, then $T$ is said to be a $G$-monotone nonexpansive mapping.  A fixed point of $T$ is any element $x \in C$ such that $T(x) =x$.  The set of all fixed points of $T$ is denoted by $Fix(T)$.
\end{definition}

\medskip

\begin{definition}\label{def-uc}  Let  $(X, \|.\|)$ be a Banach space.  Define the modulus of uniform convexity $\delta_X: (0, 2]\rightarrow [0,1]$ by
$$\delta_X(\varepsilon) = \inf\left\{1-\left\|\frac{x+y}{2}\right\|; \; \|x\| \leq 1,\ \|y\| \leq 1,\ and\ \|x-y\| \geq \varepsilon\right\}.$$
$X$ is said to be uniformly convex if $\delta_X(\varepsilon) >0$ for any $\varepsilon \in (0,2]$.
\end{definition}

Uniformly convex Banach spaces enjoy many nice geometric properties (see for example the reference \cite{beauzamy}).

\section{Iteration process for $G$-Monotone Nonexpansive Mappings}

In this section, we investigate the existence of fixed points of $G$-monotone nonexpansive mappings in $X$.  The main difficulty encountered in this setting has to do with the fact that the mappings do not have a good behavior on the entire sets.  They do have a good behavior only on connected points.  For this reason, our investigation is based on a constructive iteration  approach initiated by Krasnoselskii \cite{krasnoselskii} (see also \cite{ishikawa}).  Throughout this section we assume that $(X, \|.\|)$ is a Banach space and $G$ is a reflexive and transitive digraph defined on $X$. Moreover, we assume that $E(G)$ has property (P) and $G$-intervals are closed and convex. Recall that a $G$-interval is any of the subsets  $[a,\rightarrow) = \{x \in X; (a,x)\in E(G) \}$ and $(\leftarrow,b] = \{x \in X; (x,b)\in E(G)\}$, for any $a,b \in X$.\\

\begin{definition}\label{mann-iteration}\cite{ishikawa, krasnoselskii} Let $X$ and $G$ be as above. Let $C$ be a nonempty convex subset of $X$ and $T:C\rightarrow C$ be a $G$-monotone mapping.  Fix $x_1 \in C$. The Mann iteration process is the sequence $(x_n)$ defined by
\begin{equation} \label{mann-iteration-eqt}
x_{n+1} = t_n T(x_n) + (1-t_n) x_n,
\end{equation}
for any $n \geq 1$, where $(t_n) \subset [0,1].$
\end{definition}

\medskip

The following technical Lemmas will be useful to prove the main
result of this work.

\begin{lemma}\label{technical-lemma}  Let $X$ and $G$ be as above. Let $C$ be a nonempty convex subset of $X$ and $T:C\rightarrow C$ be a $G$-monotone mapping.  Fix $x_1 \in C$.  Consider the Mann iteration sequence $(x_n)$ defined by (\ref{mann-iteration-eqt}).
\begin{itemize}
\item[(i)] If  $(x_1 , T(x_1))\in E(G)$, then we have for any $n \geq 1$:
$$(x_n, x_{n+1}) \in E(G)\; and \; (x_{n+1}, T(x_{n})) \in E(G).$$
\item[(ii)] If  $(T(x_1), x_1)\in E(G)$, then we have for any $n \geq 1$:
 $$(x_{n+1}, x_{n}) \in E(G)\; and\; (T(x_{n}), x_{n+1}) \in E(G).$$
\end{itemize}
\end{lemma}
\begin{proof}  We will prove (i). The proof of (ii) is similar and will be omitted.
As $(x_1,T(x_1))\in E(G)$ and $(x_1,x_1) \in E(G)$, we have by property (CG)
$$\Big((1-t_{1}) x_1 + t_{1} x_{1},(1-t_{1}) x_{1} + t_{1} T(x_{1})\Big)\in E(G),$$
i.e., $(x_1,x_2)\in E(G)$.  The same argument will imply
$$\Big((1-t_{1}) x_1 + t_{1} T(x_{1}),(1-t_{1}) T(x_{1}) + t_{1} T(x_{1})\Big)\in E(G),$$
i.e., $(x_2,T(x_1))\in E(G)$.  Now assume that $(x_{n-1},x_n)\in E(G)$ and $(x_{n},T(x_{n-1}))\in E(G)$, for $n > 1$. Since $T$ is $G$-monotone, we have $(T(x_{n-1}),T(x_{n}))\in E(G)$. By transitivity of $G$, we get $(x_{n},T(x_{n}))\in E(G)$.  Hence by using property (CG), we obtain
$$\Big((1-\lambda) x_{n}+ \lambda T(x_{n}),(1-\lambda) T(x_{n})+ \lambda T(x_{n})\Big)\in E(G),$$
holds for any $\lambda\in [0,1]$, i.e., $(x_{n+1},T(x_{n}))\in E(G)$. Using the same argument, we get
$$\Big((1-\lambda) x_{n}+ \lambda x_{n},(1-\lambda) x_{n}+ \lambda T(x_{n})\Big)\in E(G),$$
holds for any $\lambda\in [0,1]$, i.e., $(x_{n},x_{n+1})\in E(G)$.  By induction, we have  $$(x_n, x_{n+1}) \in E(G)\; and \; (x_{n+1}, T(x_{n})) \in E(G).$$
for all $n \geq 1$.
\end{proof}

\begin{lemma}\label{lemma1}  Let $X$ and $G$ be as above. Let $C$ be a nonempty, closed, and convex subset of $X$ and  $T:C\rightarrow C$ be a $G$-monotone nonexpansive mapping. Let $\omega \in Fix(T)$. Let $x_1\in C$ be such that $(x_1,\omega)\in E(G)$. Let $(x_n)$ be the  Mann iteration sequence defined by (\ref{mann-iteration}).  Then we have $(x_n,\omega)\in E(G)$ for any $n\geq 1$ and $\displaystyle \lim_{n \rightarrow \infty} \|x_n-\omega\|$ exists.
\end{lemma}
\begin{proof}  Assume that $(x_1,\omega)\in E(G)$. Since $T$ is $G$-monotone, then we must have $(T(x_1), T(\omega))\in E(G)$. Since $T(\omega)=\omega$ we get $(x_1,\omega)\in E(G)$ and $(T(x_1),\omega) \in E(G)$. The property (CG) implies
$$\Big((1-\lambda) x_1 + \lambda T(x_{1}),(1-\lambda) \omega + \lambda \omega \Big)\in E(G),$$
for any $\lambda\in [0,1]$, which implies $(x_2,\omega)\in E(G)$.  By induction, we prove that $(x_n , \omega)\in E(G)$, for any $n \geq 1$. Since $T$ is $G$-monotone nonexpansive, we get
$$\|T(x_n)- \omega\| = \|T(x_n)- T(\omega)\| \leq \|x_n- \omega\|,$$
which implies
$$\begin{array}{lll}
\|x_{n+1}- \omega\| &\leq& t_n \|T(x_n)- \omega\| + (1-t_n) \|x_n- \omega\| \\
&&\\
&\leq & t_n \|x_n- \omega\| + (1-t_n) \|x_n- \omega\| = \|x_n-\omega\|,
\end{array}$$
for any $n \geq 1$.  This means that $(\|x_n-\omega\|)$ is a
decreasing sequence, which implies that $\lim\limits_{n
\rightarrow \infty} \|x_n-\omega\|$ exists.
\end{proof}

\medskip

In the general theory of nonexpansive mappings, the main property of the Mann iterative sequence is an approximate fixed point property.  Recall that $(x_n)$ is called an approximate fixed point sequence of the mapping $T$ if $\lim\limits_{n \rightarrow +\infty} \|x_n - T(x_n)\| = 0$.  We have a similar conclusion for
$G$-monotone nonexpansive mappings if we assume $X$ is uniformly convex.  Since the proof of the main result involves ultrafilters and ultrapowers of Banach spaces, let us give their definitions.  First, recall that an ultrafilter $\mathcal{U}$ over $\mathbb{N}$ is a nonempty family of subsets of $\mathbb{N}$ satisfying
\begin{itemize}
\item[(i)] ${\cal U}$ is closed under taking supersets, i.e.,
$A\in {\cal U}$ and $A\subseteq B  \Longrightarrow\ B\in {\cal
U}$;
\item[(ii)] ${\cal U}$ is closed under finite intersections, i.e.,
$A,\,B\in {\cal U}\ \Longrightarrow\ A\cap B\in {\cal U}$;
\item[(iii)] for every $A\subseteq \mathbb{N}$ precisely one of $A$ or $\mathbb{N}\backslash
A$ is in $\mathcal{U}$.
\end{itemize}

\medskip
\noindent For a Hausdorff topological space $(\Omega ,{\cal T})$, an ultrafilter ${\cal U}$ over $\mathbb{N}$ and $(x_n)_{n\in \mathbb{N}}\subseteq \Omega$, we say
$$\lim_{\cal U}\ x_n\ =\ x_0$$
if for every neighborhood $W$ of $x_0$ we have $\{n \in \mathbb{N}: x_n\in W\}\ \in\ {\cal U}$.  Such limit is unique when it exists.  It is well known that if $(\Omega ,{\cal T})$ is compact, then for any sequence $(x_n)_{n\in \mathbb{N}}\subseteq \Omega$ and any ultrafilter ${\cal U}$ over $\mathbb{N}$, the limit $\lim\limits_{\cal U}\ x_n $ exists \cite{beauzamy}.  Next, we give the definition of the ultrapower of a Banach space.

\medskip
\begin{definition}\cite{beauzamy}  Let $(X, \|.\|)$ be a Banach space and $\cal U$ an ultrafilter over $\mathbb{N}$.  Consider the Banach space
$$\ell_\infty (X)\ =\ \{(x_n)_{n\in \mathbb{N}}, \|(x_n)\|_\infty = \sup_{n\in \mathbb{N}} \ \|x_n\| < \infty \}.$$
Then $N_{\cal U}(X)\ =\ \{(x_n)_{n\in \mathbb{N}}\in \ell_\infty (X); \lim_{\cal U}\ \|x_n\| = 0\}$ is a closed linear subspace of $\ell_\infty (X)$.  The ultrapower of $X$ over ${\cal U}$ is defined to
be the Banach space quotient
$$(X)_{\cal U}\ =\ {\ell_\infty (X) / N_{\cal U}(X)},$$
with elements denoted by $(x_n)_{\cal U}$, where $(x_n)$ is a representative of the equivalence class.  The quotient norm is canonically given by
$$\|(x_n)_{\cal U}\|\ =\ \lim_{\cal U} \|x_n\|.$$
\end{definition}

\medskip\medskip
Now we are ready to state our first result.

\begin{theorem}\label{afps-montone-uc} Let $X$ and $G$ be as above. Let $C$ be a nonempty, closed, convex and bounded subset of $X$.  Let $T:C\rightarrow C$ be a $G$-monotone nonexpansive mapping. Assume $X$ is uniformly convex, and there exist $\omega \in Fix(T)$ and $x_1 \in C$  such that $(x_1,\omega)\in E(G)$.  Then we have
$$\lim_{n \rightarrow \infty} \|x_n- T(x_n)\| = 0,$$
where $(x_n)$ is the Mann iterative sequence generated by
(\ref{mann-iteration-eqt}) which starts at $x_1$, with $t_n \in
[a,b]$, for some $a > 0$ and $b < 1$.
\end{theorem}
\begin{proof} Let $\omega \in Fix(T)$ and $x_1 \in C$ such that $(x_1,\omega)\in E(G)$.  Using Lemma \ref{lemma1}, we conclude that $\displaystyle \lim_{n \rightarrow \infty} \|x_n- \omega\|$ exists.  Set $R = \displaystyle \lim_{n \rightarrow \infty} \|x_n- \omega\|$.   Moreover we have
$$\limsup_{n \rightarrow \infty} \|T(x_n)- \omega\| = \limsup_{n \rightarrow \infty}  \|T(x_n)- T(\omega)\| \leq \limsup_{n \rightarrow \infty}   \ \|x_n- \omega\| = R,$$
since $(x_n,\omega)\in E(G)$, for any $n \geq 1$, and $T$ is $G$-monotone nonepxansive. Without loss of any generality, we may assume $R > 0$.   On
the other hand, we have $$\|x_{n+1}- \omega\| \leq t_n \|T(x_n)-
\omega\| + (1-t_n) \|x_n- \omega\| \leq \|x_n- \omega\|,$$ for any
$n \geq 1$.  Let $\cal U$ be a non-trivial ultrafilter over
$\mathbb{N}$.  Then $\displaystyle \lim_{\cal U} t_n = t \in
[a,b]$. Hence
$$R =  \lim_{\cal U} \|x_{n+1}- \omega\| \leq t\  \lim_{\cal U} \|T(x_n)- \omega\| + (1-t) R \leq R.$$
Since $t \neq 0$, we get $\displaystyle \lim_{\cal U}  \|T(x_n)-
\omega\| = R$.  Consider the ultrapower $(X)_{\cal U}$ of $X$ (see
\cite{beauzamy}).  Set $\tilde{x} = (x_n)_{\cal U}$,
$\tilde{y} = (T(x_n))_{\cal U}$ and $\tilde{\omega} =
(\omega)_{\cal U}$.
  Then we have
$$\|\tilde{x}-\tilde{\omega}\|_{\cal U} = \|\tilde{y}-\tilde{\omega}\|_{\cal U} = \|t \tilde{x}+ (1-t)\tilde{y} - \tilde{\omega}\|_{\cal U}.$$
Since $t \in (0,1)$ and $X$ is uniformly convex, then $(X)_{\cal
U}$ is strictly convex (see \cite{beauzamy}) which implies $\tilde{x} = \tilde{y}$,
i.e., $\lim\limits_{n, \cal U} \|x_n - T(x_n)\| = 0$.  Since $\cal
U$ was an arbitrary non trivial ultrafilter, we conclude that
$\lim\limits_{n \rightarrow \infty} \|x_n - T(x_n)\| = 0$, which
completes the proof of Theorem \ref{afps-montone-uc}.
\end{proof}

\medskip

The conclusion of Theorem \ref{afps-montone-uc} is strongly dependent on the assumption that a fixed point of $T$ exists which is connected to $x_1$.  In fact, we may relax such assumption and obtain a similar conclusion.  First, we will need the following Proposition from \cite{goebel-kirk1983}.

\begin{proposition}  \label{goebel-kirk} Let $(X, \|.\|)$ be a Banach space.  Let $(x_n)$ and $(y_n)$ be in $X$ and $(t_n) \subset [0,1)$, such that
\begin{enumerate}
\item[(i)] $x_{n+1} = (1-t_n) x_n + t_n y_n$,
\item[(ii)] $\|y_{n+1} - y_n\| \leq \|x_{n+1}-x_n\|$,
\end{enumerate}
for any $n \in \mathbb{N}$.  Then for any $i,n \geq 1$, we have
$$\begin{array}{lll}
\displaystyle \Big(1 + \sum_{s=i}^{i+n-1} t_s\Big) \|x_i - y_i\| &\leq \|y_{i+n} - x_i\| \\
&\; +\displaystyle  \prod_{s=i}^{i+n-1} (1-t_s)^{-1}\Big[\|y_i-x_i\| - \|y_{i+n}-x_{i+n}\|\Big].\\
\end{array}$$
\end{proposition}

\medskip
The following technical lemma is crucial to the proof of our second result.

\begin{lemma}\label{goebel-kirk1983}  Let $X$ and $G$ be as above. Let $C$ be a nonempty convex subset of $X$ and $T:C\rightarrow C$ be a $G$-monotone nonexpansive mapping. Let $x_1 \in C$ be such that $(x_1,T(x_1))\in E(\widetilde{G})$.  Let $(x_n)$ be the Mann iterative sequence defined by (\ref{mann-iteration}) such that $(t_n) \subset [0,1)$. Then for any $i,n \geq 1$, we have
$$\begin{array}{lll}
\displaystyle \Big(1 + \sum_{s=i}^{i+n-1} t_s\Big) \|x_i - T(x_i)\| &\leq \|T(x_{i+n}) - x_i\| \\
&\; +\displaystyle  \prod_{s=i}^{i+n-1} (1-t_s)^{-1}\Big[\|T(x_i)-x_i\| - \|T(x_{i+n})-x_{i+n}\|\Big].\\
\end{array}$$
\end{lemma}
\begin{proof} Without loss of any generality, we may assume $(x_1,T(x_1))\in E(G)$.  Lemma \ref{technical-lemma} implies that $(x_n, x_{n+1}) \in E(G)$ for any $n \geq 1$.  Since $T$ is a $G$-monotone nonexpansive mapping, we get
$$\|T(x_{n+1}) - T(x_n)\| \leq \|x_{n+1}-x_n\|,$$
and $(T(x_n), T(x_{n+1})) \in E(G)$, for any $n \geq 1$. Moreover from the definition of $(x_n)$ we have $\|x_{n+1} - x_n \| = t_n \|x_n - T(x_n)\|$, for any $n \geq 1$.  Therefore all the assumptions of Proposition \ref{goebel-kirk} are satisfied, where $(y_n) = (T(x_n))$, which implies the conclusion of Lemma \ref{goebel-kirk1983}.
\end{proof}

Using this lemma, we have a similar conclusion to Theorem
\ref{afps-montone-uc} with less stringent assumptions.  This
result is similar to the one found in \cite{ishikawa}.

\begin{theorem}\label{gk-afps-monotone}  Let $X$ and $G$ be as above. Let $C$ be a nonempty, closed, convex and bounded subset of $X$ and $T:C\rightarrow C$ be a $G$-monotone nonexpansive mapping. Let $x_1 \in C$ be such that $(x_1,T(x_1))\in E(\widetilde{G})$.  Let $(x_n)$ be the Mann iterative sequence defined by (\ref{mann-iteration}) such that $(t_n) \subset [a,b]$, with $a > 0$ and $b < 1$. Then we have $\lim\limits_{n \rightarrow +\infty} \|x_n - T(x_n)\| = 0$.
\end{theorem}
\begin{proof}  First note that the sequence $(\|x_n-T(x_n)\|)$ is decreasing.  Indeed, we have
$$\begin{array}{lll}
\|x_{n+1} - T(x_{n+1})\| & = & \|(1-t_n)x_n + t_n T(x_n) - T(x_{n+1})\| \\
&=& \|(1-t_n)(x_n - T(x_n)) + T(x_n) - T(x_{n+1})\|\\
&\leq& (1-t_n)\|x_n - T(x_n)\| + \|T(x_n) - T(x_{n+1})\|\\
&\leq& (1-t_n)\|x_n - T(x_n)\| + \|x_n - x_{n+1}\|\\
&=& (1-t_n)\|x_n - T(x_n)\| + t_n \|x_n - T(x_n)\|\\
&=& \|x_n - T(x_n)\|,
\end{array}$$
for any $n \geq 1$.  Set $\lim\limits_{n \rightarrow +\infty}
\|x_n - T(x_n)\| = R$.  Next we note that we have:
$$\left\{\begin{array}{lll}
(1+n a) \leq 1 + \sum\limits_{s=i}^{i+n-1} t_s,\\
\prod\limits_{s=i}^{i+n-1} (1-t_s)^{-1} \leq (1-b)^{-n},\\
\|T(x_{i+n}) - x_i\| \leq \delta(C) = \sup \{\|x-y\|; x, y \in C\}.
\end{array}\right.$$
Hence the main inequality obtained in Lemma \ref{goebel-kirk1983} implies
$$(1+n a) \|x_i - T(x_i)\| \leq \delta(C) + (1-b)^{-n} \Big[\|T(x_i)-x_i\| - \|T(x_{i+n})-x_{i+n}\|\Big],$$
for any $i, n \geq 1$.  If we let $i \rightarrow +\infty$, we get $(1+n a) R \leq \delta(C)$, for any $n \geq 1$.  Hence
$$R \leq \frac{\delta(C)}{(1+na)},$$
holds for any $n \geq 1$.  Clearly this will imply $R =0$, i.e.,
$$\lim\limits_{n \rightarrow+\infty} \|x_n - T(x_n)\| = 0.$$
\end{proof}

\medskip

Before we state the main fixed point result of this work, let us
recall the definition of the weak-Opial condition.

\begin{definition}\cite{opial}  Let $(X, \|.\|)$ be a Banach space.  We will say that $X$ satisfies the weak-Opial condition if for any sequence $(y_n)$ which converges weakly to $y$, we have
$$\liminf_{n \rightarrow \infty} \|y_n - y\|   < \liminf_{n \rightarrow \infty} \|y_n - z\|,$$
for any $z \in X$ such that $y \neq z$.
\end{definition}

\medskip

\begin{theorem}\label{mann-fpp-uc} Let $(X, \|.\|)$ be a Banach space which satisfies the weak-Opial condition and $G$ be the directed reflexive and transitive digraph defined on $X$. Let $C$ be a nonempty weakly compact convex subset of $X$ and $T:C\rightarrow C$ be a $G$-monotone nonexpansive mapping. Assume there exists $x_1 \in C$ such that $(x_1,T(x_1))\in E(\widetilde{G})$.  Let $(x_n)$ be the Mann iterative sequence defined by (\ref{mann-iteration}) such that $(t_n) \subset [a,b]$, with $a > 0$ and $b < 1$.  Then $(x_n)$ is weakly convergent to $x$ which is a fixed point of $T$, i.e., $T(x) = x$.  Moreover $(x_1,x)\in E(G)$. \end{theorem}
\begin{proof}  Without loss of any generality, we may assume that $(x_1, T(x_1))\in E(G)$.  From the previous Lemmas, we know that
$$\lim\limits_{n \rightarrow +\infty} \|x_n - T(x_n)\| = 0.$$
Let $\omega_1$ be a weak-cluster point of $(x_n)$.  Then there exists $(x_{\varphi(n)})$ a subsequence of $(x_n)$ which converges weakly to $\omega_1$.  From the assumptions assumed, we know that $(x_n,\omega_1) \in E(G)$, for any $n \geq 1$.  Let us prove that $\omega_1$ is a fixed point of $T$.  Since
$$\Big|\|T(\omega_1) - x_{\varphi(n)}\| - \|T(\omega_1) - T(x_{\varphi(n)})\| \Big| \leq \|T(x_{\varphi(n)}) - x_{\varphi(n)}\|,$$
for any $n\in \mathbb{N}$, we conclude that
$$\liminf\limits_{n \rightarrow \infty} \|T(\omega_1) - x_{\varphi(n)}\| = \liminf\limits_{n \rightarrow \infty} \|T(\omega_1) - T(x_{\varphi(n)})\|.$$
Hence, we have
$$\liminf\limits_{n \rightarrow \infty} \|T(\omega_1) - x_{\varphi(n)}\| = \liminf\limits_{n \rightarrow \infty} \|T(\omega_1) - T(x_{\varphi(n)})\| \leq \liminf\limits_{n \rightarrow \infty} \|\omega_1 - x_{\varphi(n)}\|.$$
The weak-Opial property implies that $T(\omega_1) = \omega_1$.  Let $\omega_2$ be another weak-cluster point of $(x_n)$.  Again there exists a subsequence $(x_{\psi(n)})$ of $(x_n)$ which converges weakly to $\omega_2$.  The same argument above shows that $\omega_2$ is also a fixed point of $T$.  In this case, we have seen that $(\|x_n - \omega_i\|)$ are convergent for $i = 1, 2$.  Let us show that $\omega_1 = \omega_2$.  Assume not, i.e., $\omega_1 \neq \omega_2$.  Then we have
$$\begin{array}{lll}
\liminf\limits_{n \rightarrow \infty} \|\omega_2 - x_{\varphi(n)}\| &=& \lim\limits_{n \rightarrow \infty}\|\omega_2 - x_{n}\|\\
&=& \lim\limits_{n \rightarrow \infty}\|\omega_2 - x_{\psi(n)}\|\\
&< & \lim\limits_{n \rightarrow \infty}\|\omega_1 - x_{\psi(n)}\|\\
&=& \lim\limits_{n \rightarrow \infty}\|\omega_1 - x_{n}\|\\
&=& \lim\limits_{n \rightarrow \infty}\|\omega_1 - x_{\varphi(n)}\|\\
\end{array}$$
which is a contradiction with the fact that $(x_{\varphi(n)})$ converges weakly to $\omega_1$ and the weak-Opial property.  Therefore we must have $\omega_1 = \omega_2$.  This clearly implies that $(x_n)$ is weakly convergent and its weak limit is a fixed point of $T$.
\end{proof}

\noindent {\large{\textbf{Acknowledgement}}}
\medskip \\
\noindent The author would like to acknowledge the support provided by the Deanship of Scientific Research at King Fahd University of Petroleum \& Minerals for funding this work through project No. IP142-MATH-111.

\end{document}